\documentclass[10pt,reqno]{amsart}
\usepackage{geometry}
\usepackage[dvips]{graphics}
\usepackage[english]{babel}
\usepackage{amsfonts}
\usepackage{amsthm}
\usepackage{amsmath}
\usepackage{amscd}
\usepackage[active]{srcltx} % SRC Specials: DVI [Inverse] Search
\usepackage{latexsym}
\usepackage{amssymb}
\usepackage[latin1]{inputenc}
\numberwithin{equation}{section}

\newtheorem{thm}{Theorem}[section]
\newtheorem{prop}[thm]{Proposition}
\newtheorem{cor}[thm]{Corollary}
\newtheorem*{cor*}{Corollary}
\newtheorem{lema}[thm]{Lemma}
\newtheorem*{lema*}{Lemma}

\theoremstyle{definition}

\newtheorem{prob}{Problem}
\newtheorem*{prob*}{Problem}

\newtheorem*{Def}{Definition}

\newtheorem{obs}[thm]{Remark}

\newtheorem*{obs*}{Remark}
\newtheorem*{thm*}{Theorem}
\newtheorem*{prop*}{Proposition}

\newcommand{\PI}[2]{\left\langle \,#1 , #2\, \right\rangle}

\newcommand{\K}[2]{[ \,#1 , #2\, ]}

\newcommand{\ra}{\rightarrow}

\newcommand{\x}{\times}

\newcommand{\CC}{\mathbb{C}}
\newcommand{\RR}{\mathbb{R}}
\newcommand{\NN}{\mathbb{N}}
\newcommand{\St}{\mathcal{S}}

\newcommand{\N}{\mathcal{N}}

\newcommand{\M}{\mathcal{M}}

\newcommand{\HH}{\mathcal{H}}
\newcommand{\KK}{\mathcal{K}}
\newcommand{\EE}{\mathcal{E}}
\newcommand{\mc}[1]{\mathcal{#1}}
\newcommand{\parentesis}[1]{\left( \,#1\, \right)}
\newcommand{\set}[1]{\left\{ \,#1\, \right\}}

\newcommand{\ol}{\overline}

\newcommand{\ort}{[\bot]}

\newcommand{\noi}{\noindent}
\newcommand{\sdo}{[\dotplus]}

\newcommand{\CV}{\mathcal{C}_{V}}

\newcommand{\PK}{\mathcal{P}^+(\KK)}

\begin{document}

\title[Regularization of an Abstract Interpolation Problem]{Regularization of an Indefinite Abstract Interpolation Problem with a Quadratic Constraint}

\author[S.~Gonzalez Zerbo]{Santiago Gonzalez Zerbo} 
\address{Departamento de Matem\'{a}tica, Facultad de Ingenier\'{\i}a -- Universidad de Buenos Aires, and Instituto Argentino de Matem\'{a}\-tica ``Alberto P. Calder\'{o}n'' (CONICET), Saavedra 15 (1083) Buenos Aires, Argentina}
\email{sgzerbo@fi.uba.ar}

\author[A. Maestripieri]{Alejandra Maestripieri} 
\address{Departamento de Matem\'{a}tica, Facultad de Ingenier\'{\i}a -- Universidad de Buenos Aires, and Instituto Argentino de Matem\'{a}\-tica ``Alberto P. Calder\'{o}n'' (CONICET), Saavedra 15 (1083) Buenos Aires, Argentina}
\email{amaestri@fi.uba.ar}

\author[F. Mart\'{\i}nez Per\'{\i}a]{Francisco Mart\'{\i}nez Per\'{\i}a}
\address{Centro de Matem\'{a}tica de La Plata (CeMaLP) -- FCE-UNLP, La Plata, Argentina \\
and Instituto Argentino de Matem\'{a}tica ``Alberto P. Calder\'{o}n'' (CONICET), Saavedra 15 (1083) Buenos Aires, Argentina}
\email{francisco@mate.unlp.edu.ar}

\begin{abstract}
Along this work we study an indefinite abstract smoothing problem. After establishing necessary and sufficient conditions for the existence of solutions to this problem, the set of admissible parameters
is discussed in detail. Then, its relationship with a linearly constrained interpolation problem is analyzed.
\end{abstract}


\maketitle

\section{Introduction}
In the sixties, a Hilbert space formulation of spline functions, known as abstract splines, was introduced by M. Atteia \cite{[Att1]} and extended by several authors, see e.g.
\cite{[AL], [dB], [Lau], [S1]}. Abstract splines are the solution to an abstract interpolation problem, which has a natural generalization to Krein spaces. Spline functions in indefinite metric
spaces have already been studied in \cite{Canu learning,Loosli} to solve numerical aspects related to learning theory problems. Although the problems presented there are different from those studied in this work, they are closely related.

 Given a Hilbert space $\HH$, and Krein spaces $\KK$ and $\EE$, consider (bounded) surjective  operators $T : \HH \ra \KK$ and $V : \HH \ra \EE$. The abstract interpolation problem in Krein spaces
 can be stated as follows: given $z_0\in \EE$,
\begin{equation}\label{QCQP}
\textit{\normalfont{minimize} }\K{Tx}{Tx},\textit{ \normalfont{subject to} }\K{Vx-z_0}{Vx-z_0}=0,
\end{equation}
and if the minimum exists, find the set of arguments at which it is attained.

Since $\K{\cdot}{\cdot}_\KK$ and $\K{\cdot}{\cdot}_\EE$ are indefinite inner products, the above one is a quadratically constrained quadratic programming (QCQP) problem, where the objective
function $x\mapsto \K{Tx}{Tx}$ is not convex while the function defining
the equality constraint $x\mapsto \K{Vx-z_0}{Vx-z_0}$ is sign indefinite.  

Given $x_0\in\HH$ such that $Vx_0=z_0$, the condition $\K{Vx-z_0}{Vx-z_0}=0$ becomes $x\in x_0 + \CV$, 
where $\CV=\set{u\in\HH: \K{Vu}{Vu}=0}$ is the set of neutral elements of the quadratic form $x\mapsto \K{Vx}{Vx}$.
Hence, \eqref{QCQP} can be restated as: analyze the existence of 
\begin{equation*}
\min_{x\in x_0 + \CV} \K{Tx}{Tx}.
\end{equation*}
and if the minimum exists, find the set of arguments at which it is attained.

If $V^\#V$ is a positive (or negative) semidefinite operator in $\HH$ then $\CV$ coincides with $N(V)$,
and \eqref{QCQP} becomes the interpolation problem previously studied in \cite{GMMP10}.
But, if $V^\#V$ is \emph{indefinite}, the set $\CV$ is strictly larger than $N(V)$.
Therefore, in the main part of this work $V^\#V$ is assumed to be indefinite, i.e. neither
positive nor negative semidefinite. The particular case of a semidefinite $V^\#V$ is only considered in Section \ref{semidef}. 

One of the main drawbacks for tackling this QCQP problem is that $\CV$ is not a convex set. Moreover, the convex hull of $\CV$ is the complete Hilbert space $\HH$, thus replacing $\CV$ by its convex hull trivializes the problem.

\medskip 

The classical penalization approach known as the Tikhonov regularization can be applied to this problem in order to obtain an associated optimization problem over the Hilbert space $\HH$. Indeed, the aim of this work is to study the following
generalization of the abstract smoothing problem \cite{[Att4]}:

\begin{prob*}
Given $\rho \in\RR\setminus\{0\}$ and a fixed $z_0\in \EE$, analyze the existence of
\begin{equation*}
\min_{x\in\HH}\parentesis{\K{Tx}{Tx}_\KK+\rho\K{Vx-z_0}{Vx-z_0}_\EE},
\end{equation*}
and if the minimum exists, find the set of arguments at which it is attained.
\end{prob*}

The advantage of the above regularized problem is that it can be restated as an indefinite least-squares problem. These problems have been thoroughly studied before, both in
finite-dimensional spaces \cite{H1,H2,Sayed,CHGS98,HSK99,Patel,Bojanczyk} and in infinite dimensional Krein spaces \cite{Bognar,GMMP10b,GMMP16}. Although the existence of solutions to Problem
\ref{pb smooth} as well as descriptions of the set of solutions can be derived from the corresponding indefinite least-squares problem, it is desirable to reformulate these conditions in terms of
the original data $T$, $V$, $\rho$ and $z_0$.

The indefinite abstract smoothing problem in Krein spaces was initially studied in \cite{GMMP10}, but for a linear constraint. Also, in \cite{Canu splines} another version of this abstract smoothing problem was studied, but instead of minimizing $F_\rho (x)=\K{Tx}{Tx}_\KK + \rho \|Vx-z_0\|^2_\EE$, the authors were interested in stabilizing (i.e. finding the critical points of) this function. 

\medskip

The contents of the paper are organized as follows. Section 2 explains the notations, and it also contains a compilation of the basic terminology and results on Krein spaces  used along this work. 

Section 3 deals with the indefinite abstract smoothing problem. After establishing necessary and sufficient conditions for the existence of solutions to this problem, the set of admissible parameters
is discussed in detail.

Finally, Section 4 is devoted to analyzing the close relationship between the indefinite abstract smoothing problem and a particular version of the indefinite abstract interpolation problem, where the
quadratic constraint in \eqref{QCQP} is replaced by a linear constraint.
This problem coincides with the one previously
considered in \cite{GMMP10}. We focus our attention on the situations in which a set of interpolating splines is also a subset of the solutions to a certain smoothing interpolation problem, or, even
further, in which an indefinite interpolation problem can be posed as an indefinite smoothing problem, and viceversa.

\section{Preliminaries}

Along this work $\HH$ denotes a complex (separable) Hilbert space. If $\mc{K}$ is another Hilbert space then $\mc{L}(\HH, \KK)$ is the vector space of bounded linear operators from $\HH$ into $\KK$ and
$\mc{L}(\HH)=\mc{L}(\HH,\HH)$ stands for the algebra of bounded linear operators in $\HH$.
%and denote by $\Q$ the set of (oblique) projections, i.e. $\Q=\{Q\in L(\EE)\; : \; Q^2=Q\}$. 

If $T\in \mc{L}(\HH, \KK)$ then $R(T)$ stands for the range of $T$ and $N(T)$ for its nullspace. The Moore-Penrose inverse $T^\dag$ of an operator $T\in \mathcal{L}(\HH,\KK)$ is the densely defined (not necessarily bounded) linear operator 
\[
T^\dag: R(T) \dotplus R(T)^\bot \ra \HH,
\]
defined as follows: for $y\in R(T)$ and $z\in R(T)^\bot$, $T^\dag(y+z)=x$ where $x\in N(T)^\bot$ is the (unique) vector that satisfies $Tx=y$.

Recall that the Moore-Penrose inverse $T^\dag$ is bounded if and only if $T$ has a closed range. 
For detailed expositions on the Moore-Penrose inverse, see \cite{BG,Nashed}.

The following well-known result about the product of closed range operators \cite{Bouldin, Izumino} is frequently used along the paper.

\begin{prop}\label{R(AB)}
Given Hilbert spaces $\HH_1$, $\HH_2$ and $\KK$, let $A\in \mc{L}(\KK,\HH_2)$ and $B\in \mc{L}(\HH_1,\KK)$ have closed ranges. Then, $AB\in \mc{L}(\HH_1,\HH_2)$ has closed range if and only if $R(B) + N(A)$ is closed in $\KK$. %$c(R(B), N(A))<1$.
\end{prop}

\subsection{Krein spaces}

In what follows we present the standard notation and some basic results on Krein spaces. For a complete exposition on the subject (and the proofs of the results below) see
\cite{Bognar,Azizov,Ando,Dritschel 1,Rovnyak}.

\medskip

An indefinite inner product space $(\mc{F}, \K{\cdot}{\cdot})$ is a (complex) vector space $\mc{F}$ endowed with a Hermitian sesquilinear form $\K{\cdot}{\cdot}: \mc{F}\x\mc{F} \ra \CC$.

A vector $x\in\mc{F}$ is {\it positive}, {\it negative}, or {\it neutral} if $\K{x}{x}>0$, $\K{x}{x}<0$, or $\K{x}{x}=0$, respectively.
The set of positive vectors in $\mc{F}$ is denoted by $\mathcal{P}^{++}(\mc{F})$, and the set of {\it nonnegative} vectors in $\mc{F}$
by $\mathcal{P}^+(\mc{F})$. The sets of negative, nonpositive and neutral vectors in $\mc{F}$ are defined analogously, and they are denoted by $\mathcal{P}^{--}(\mc{F})$, $\mathcal{P}^{-}(\mc{F})$, and $\mathcal{P}^0(\mc{F})$, respectively.

Likewise, a subspace $\M$ of $\mc{F}$ is {\it positive} if every $x\in\M$, $x\neq0$ is a positive vector in $\mc{F}$; and it is
{\it nonnegative} if $\K{x}{x} \geq0$ for every $x\in\M$. Negative, nonpositive and neutral subspaces are defined mutatis mutandis.

If $\St$ is a subset of an indefinite inner product space $\mc{F}$, the \emph{orthogonal companion} to $\St$ is defined by 
\[
\St^{\ort}=\set{ x\in\mc{F} : \K{x}{s}=0 \; \text{for every $s\in\St$}}.
\]
It is easy to see that $\St^{\ort}$ is always a subspace of $\mc{F}$.

\begin{Def}
An indefinite inner product space $(\HH, \K{\cdot}{\cdot})$ is a \emph{Krein space} if it can be decomposed as a direct (orthogonal) sum of a Hilbert space and an anti Hilbert space, i.e. there
exist subspaces $\HH_\pm$ of $\HH$ such that $(\HH_+, \K{\cdot}{\cdot})$ and $(\HH_-, -\K{\cdot}{\cdot})$ are Hilbert spaces,
\begin{equation}\label{desc cano}
\HH=\HH_+ \dotplus \HH_-,
\end{equation}
and $\HH_+$ is orthogonal to $\HH_-$ with respect to the indefinite inner product. Sometimes we use the notation $\K{\cdot}{\cdot}_\HH$ instead of $\K{\cdot}{\cdot}$ to emphasize the Krein space considered.
\end{Def}

A pair of subspaces $\HH_\pm$ as in \eqref{desc cano} is called a \emph{fundamental decomposition} of $\HH$. Given a Krein space $\HH$ and a fundamental decomposition
$\HH=\HH_+\dotplus \HH_-$, the direct (orthogonal) sum of the Hilbert spaces $(\HH_+, \K{\cdot}{\cdot})$ and $(\HH_-, -\K{\cdot}{\cdot})$ is denoted
by $(\HH,\PI{\cdot}{\cdot})$.

If $\HH=\HH_+ \dotplus \HH_-$ and $\HH=\HH'_+ \dotplus \HH'_-$ are two different fundamental decompositions of $\HH$, the corresponding associated inner products $\PI{\cdot}{\cdot}$ and
$\PI{\cdot}{\cdot}'$ turn out to be equivalent on $\HH$. Therefore, the norm topology on $\HH$ does not depend on the chosen fundamental decomposition.

\medskip

If $(\HH,\K{\cdot}{\cdot}_\HH)$ and $(\KK,\K{\cdot}{\cdot}_\KK)$ are Krein spaces, $\mc{L}(\HH, \KK)$ stands for the vector space of linear transformations which are
bounded with respect to any of the associated Hilbert spaces $(\HH,\PI{\cdot}{\cdot}_\HH)$ and $(\KK,\PI{\cdot}{\cdot}_\KK)$. 
Given $T\in \mathcal{L}(\HH,\KK)$, the adjoint operator of $T$ (in the Krein spaces sense) is the unique operator $T^\#\in \mc{L}(\KK, \HH)$ such that
\[
\K{Tx}{y}_\KK=\K{x}{T^\#y}_\HH, \ \ \ \ x\in\HH ,y\in\KK.
\]
%An operator $T\in \mathcal{L}(\HH)$ is selfadjoint if $T=T^\#$. 

We will frequently use that if $T\in \mc{L}(\HH,\KK)$ and $\M$ is a closed subspace of $\KK$ then 
\begin{equation}\label{preimag}
T^\#(\M)^{\ort_\HH}= T^{-1}(\M^{\ort_\KK}).
\end{equation}

Given a subspace $\M$ of a Krein space $\HH$, the {\it isotropic part} of $\M$ is defined by $\M^\circ:=\M\cap\M^{\ort}$. Then, $\M$ is {\it nondegenerate} if $\M^\circ=\set{0}$.

A subspace $\M$ of a Krein space $\HH$ is {\it pseudo-regular} if $\M+\M^{\ort}$ is a closed subspace of $\HH$, and it is {\it regular} if $\M+\M^{\ort}=\HH$. Regular subspaces are examples of
nondegenerate subspaces, but pseudo-regular subspaces can be degenerate ones. However, if $\M$ is a pseudo-regular subspace then there exists a regular subspace $\mc{R}$ such that 
\begin{equation}\label{desc}
\M=\M^\circ \sdo\ \mc{R},
\end{equation} 
where $\sdo$ stands for the direct orthogonal sum with respect to the indefinite inner product $\K{\cdot}{\cdot}$.
In fact, any closed subspace $\mc{R}$ such that $\M=\M^\circ \dotplus \mc{R}$ satisfies \eqref{desc} and
%and every closed subspace $\mc{R}$ satisfying \eqref{desc} 
turns out to be a regular subspace of $\HH$, see e.g. \cite{G84}. 
Note that a subspace $\M$ is regular if and only if it is pseudo-regular and nondegenerate.

The following propositions can be found in \cite[Lemma 3.4]{GMMP16} and \cite[Chapter 1, \S 7]{Azizov}, respectively.

\begin{prop}\label{preeliminares pseudo regular}
Given Krein spaces $\HH$ and $\KK$, let $T\in \mc{L}(\HH,\KK)$ with closed range. Then, $R(T)$ is pseudo-regular if and only if $R(T^\#T)$ is closed. 
\end{prop}

A subspace $\M$ of a Krein space $(\HH,\K{\cdot}{\cdot})$ is {\it uniformly positive} if %, fixing any inner product induced norm $\|\cdot\|$, 
there exists $\alpha>0$ such that 
\[
\K{x}{x}\geq\alpha\|x\|^2\quad \text{ for every $x\in\M$},
\]
where $\|\cdot\|$ is the norm of any associated Hilbert space. Uniformly negative subspaces are defined mutatis mutandis.

\begin{prop}\label{preeliminares regular}
Let $\M$ be a subspace of a Krein space $\HH$. Then, $\M$ is closed and uniformly positive (resp. negative) if and only if $\M$ is regular and nonegative (resp. nonpositive).
\end{prop}

The following result, taken from \cite[Chapter 1, \S 1]{Azizov}, deals with indefinite inner product spaces which are not necessarily Krein spaces. It will be one of the main tools to study the problem posed in \eqref{QCQP}.

\begin{lema}\label{mapeo}
Suppose that $(\mc{F}_1,\K{\,}{\,}_1)$ is an indefinite inner product space and $(\mc{F}_2,\K{\,}{\,}_2)$ is an arbitrary inner product space. If the linear mapping $T:\mc{F}_1\ra \mc{F}_2$ satisfies
that $T(\mc{P}_1^0)\subset \mc{P}_2^+$ then, for every $y\in \mc{P}_1^{--}$ and $z\in\mc{P}_1^{++}$,
\begin{equation*}
	\frac{\K{Ty}{Ty}_2}{\K{y}{y}_1}\leq \frac{\K{Tz}{Tz}_2}{\K{z}{z}_1}.
\end{equation*}
Moreover, under these conditions,
\[
\mu_+(T):=\inf_{z\in \mc{P}_1^{++}} \frac{\K{Tz}{Tz}_2}{\K{z}{z}_1} \ \ \ \text{and} \ \ \ \mu_-(T):=\sup_{y\in \mc{P}_1^{--}} \frac{\K{Ty}{Ty}_2}{\K{y}{y}_1}
\]
exist, $\mu_-(T)\leq \mu_+(T)$ and, for any $\mu\in [\mu_-(T), \mu_+(T)]$ the following inequality holds
\begin{equation*}
\K{Tx}{Tx}_2\geq \mu \K{x}{x}_1 \ \ \ \text{for every $x\in\mc{F}_1$}.
\end{equation*}
\end{lema}

Observe that if $0\in[\mu_-,\mu_+]$ then $T^\#T$ is a positive semidefinite operator.

\section{Indefinite smoothing splines}

From now on  $(\EE,\K{\cdot}{\cdot}_\EE)$ and $(\KK,\K{\cdot}{\cdot}_\KK)$ are Krein spaces and $\HH$ is a Hilbert space, $T\in L(\HH,\KK)$ and $V\in L(\HH,\EE)$ are two given surjective operators,
and $V^\#V$ is assumed to be indefinite.

Consider the Tikhonov regularization of the problem posed in \eqref{QCQP}:
% \cite{[Att4]}:

\begin{prob}\label{pb smooth}
Given $\rho \in\RR$ and fixed $z_0\in \EE$, analyze the existence of
\begin{equation*}
\min_{x\in\HH}\parentesis{\K{Tx}{Tx}_\KK+\rho\K{Vx-z_0}{Vx-z_0}_\EE},
\end{equation*}
and if the minimum exists, find the set of arguments at which it is attained.
\end{prob}

Hereafter, we address it as the {\it indefinite abstract smoothing problem}.

\medskip
%
%\begin{Def}
%Any element $x_0\in \HH$ that is a solution to Problem \ref{pb smooth} is called a \emph{$(T,V,\rho)$-smoothing spline} to $z_0\in\EE$. 
%The set of $(T,V,\rho)$-smoothing splines to $z_0$ is denoted by $sm(T,V,\rho,z_0)$.
%\end{Def}
%
 %\medskip
 
%Given a constant $\rho\neq0$,  d
As it was mentioned in the Introduction, Problem \ref{pb smooth} can be restated as an indefinite least-squares problem.
Define an indefinite inner product on $\KK\x \EE$ by:
\begin{equation}\label{met indef para KxE}
\K{(y,z)}{(y',z')}_{\rho}=\K{y}{y'}_\KK + \rho\K{z}{z'}_\EE,\quad\quad \textrm{$y,y'\in \KK$ and $z,z'\in\EE$}.
\end{equation}
It is easy to see that the space $(\KK\x\EE,\K{\cdot}{\cdot}_\rho)$ is a Krein space if and only if $\rho\neq0$. 
Also, defining the operator $L: \HH \ra \KK\x\EE$ by
\begin{equation}\label{L}
Lx=(Tx,Vx), \ \ \ \ x\in \HH,
\end{equation}
it is immediate that Problem \ref{pb smooth} is equivalent to the following:

\begin{prob}\label{pb smooth 1}
Given $\rho \in\RR$ and a fixed $z_0\in \EE$, analize the existence of
\begin{equation}\label{cuadr min}
\min_{x\in\HH} \K{Lx-(0,z_0)}{Lx-(0,z_0)}_{\rho},
\end{equation}
and if the minimum exists, find the set of arguments at which it is attained.
\end{prob}

Indefinite least-squares problems have been thoroughly studied before, see e.g. \cite{Bojanczyk,GMMP10b,GMMP16} and the references therein. It is well-known that, if $\rho\neq0$, there exists a solution to
\eqref{cuadr min} if and only if $R(L)$ is nonnegative in $(\KK\x\EE,\K{\cdot}{\cdot}_\rho)$ and $(0,z_0)\in R(L)+R(L)^{\ort}$, see \cite[Thm. 8.4]{Bognar}. In this case, 
\begin{equation*}
	\tilde{x}\in\HH \ \text{is a solution to Problem \ref{pb smooth 1}} \ \ \ \Longleftrightarrow \ \ \ L\tilde{x} - (0,z_0)\in R(L)^{\ort}.
\end{equation*}
Moreover, if $\tilde{x}$ is a particular solution, the set of solutions is the affine manifold
\[
\tilde{x} + N(L^\#L),
\]  
where $L^\#$ stands for the adjoint operator of $L$, see the Preliminaries.

% with respect to the indefinite inner product $\K{\cdot}{\cdot}_\rho$ in $\KK\x\EE$

\medskip
Next, we show that the existence of a constant $\rho\neq 0$ such that $R(L)$ is a nonnegative subspace of $(\KK\x\EE,\K{\cdot}{\cdot}_\rho)$ is completely determined by the action of $T$ on the
vectors of $\CV$.
%definiteness of the range of $L$ with respect to the indefinite inner product 

\begin{prop}\label{prop cono positivo}
The following conditions are equivalent:
\begin{enumerate}
	\item $T(\CV)\subseteq\PK$;
	\item there exists $\rho\in\RR$ such that $T^\#T+\rho V^\#V\in \mathcal{L}(\HH)^+$.
\end{enumerate}

\end{prop}

\begin{proof}

By Lemma \ref{mapeo},
if $T(\CV)\subseteq \mc{P}^+(\KK)$ then there exists $\mu\in\RR$ such that 
\[
\K{Tx}{Tx}_\KK\geq \mu \K{Vx}{Vx}_\EE \ \ \ \text{ for every $x\in\HH$}.
\]
But this is equivalent to saying that $T^\#T -\mu V^\#V\in \mathcal{L}(\HH)^+$. Thus, (2) holds considering $\rho=-\mu$. 

Conversely, suppose that $T^\#T +\rho V^\#V\in \mathcal{L}(\HH)^+$ for some $\rho\in\RR$. Then, for every $x\in\CV$, % it holds that
\[
0\leq\PI{(T^\#T+\rho V^\#V)x}{x}=\K{Tx}{Tx}+\rho\K{Vx}{Vx}=\K{Tx}{Tx}. \qedhere
\]
\end{proof}

%Denote by $\mc{P}_V^{++}(\EE)$ ($\mc{P}_V^{--}(\EE)$) the set of positive (negative) vectors of $\EE$ with respect to the indefinite inner product $\K{V\cdot}{V\cdot}$. 
If $T(\CV)\subseteq\PK$, then
Lemma \ref{mapeo} states that 
\begin{equation}\label{rhos}
\rho_-:=-\inf_{\{x:\ \K{Vx}{Vx}>0\}}\frac{\K{Tx}{Tx}}{\K{Vx}{Vx}}\ \ \ \text{and}\ \ \ \rho_+:=-\sup_{\{x:\ \K{Vx}{Vx}<0\}}\frac{\K{Tx}{Tx}}{\K{Vx}{Vx}}
\end{equation}
exist, and $\rho_-\leq\rho_+$. Furthermore, $T^\#T+\rho V^\#V\in\mc{L}(\HH)^+$ if and only if $\rho\in[\rho_-,\rho_+]$.
Since $(\KK\x\EE,\K{\cdot}{\cdot}_\rho)$ is a Krein space if and only if $\rho\neq0$, the argument above leads to the following corollary. %considering some $\rho\neq0$.
\begin{cor}\label{cor_rhos}
Assume that $T(\CV)\subseteq\PK$, and consider $\rho_\pm$ defined by \eqref{rhos}. If $\rho\neq0$ then the following conditions are equivalent:
\begin{enumerate}
\item $T^\#T+\rho V^\#V\in\mc{L}(\HH)^+$;
\item $\rho\in[\rho_-,\rho_+]$;
\item $R(L)$ is a nonnegative subspace of $(\KK\x\EE,\K{\cdot}{\cdot}_\rho)$.
\end{enumerate}
\end{cor}

\begin{proof}
The equivalence between (1) and (2) is derived from Lemma \ref{mapeo}.
\smallskip

To establish their equivalence with (3), let us calculate the adjoint operator of $L$: 
%with respect to the indefinite inner product $\K{\cdot}{\cdot}_\rho$ in $\KK\x\EE$ defined in \eqref{met indef para KxE}.
given $x\in\HH$ and $(y,z)\in\KK\x\EE$, 
\begin{align*}
\K{Lx}{(y,z)}_\rho &=  \K{Tx}{y}_\KK + \rho\K{Vx}{z}_\EE=\PI{x}{T^\#y} + \rho\PI{x}{V^\#z} \\ &=\PI{x}{T^\#y + \rho V^\#z}.
\end{align*}
Hence, 
\[
L^\#(y,z)=T^\#y + \rho V^\#z, \quad (y,z)\in\KK\x\EE,
\]
and it is immediate that $L^\#L=T^\#T+ \rho V^\#V$. Therefore, $\K{Lx}{Lx}_\rho\geq0$ for every $x\in\HH$ 
%(i.e. $R(L)$ is a nonnegative subspace of $(\KK\x\EE,\K{\cdot}{\cdot}_\rho)$)
if and only if $T^\#T+\rho V^\#V\in\mc{L}(\HH)^+$.
\end{proof}

%If $T(\CV)\subseteq\PK$, a particular situation takes place when $\rho_+=\rho_- =0$. In this case, there is no admissible $\rho\neq 0$ for  Problem \ref{pb smooth}.
%%the only admissible $\rho$ for Problem \ref{pb smooth} to admit a solution is $\rho=0$. Thus,
%%the only indefinite smoothing problem that can be established reduces to analyzing of the existence of
%%\[
%%\min_{x\in\HH}\K{Tx}{Tx}.
%%\] 
%%Since we are interested in studying the more general case, f
%From now, if $T(\CV)\subseteq\PK$ we assume that $[\rho_-,\rho_+]\neq\{0\}$ and we consider a fixed $\rho\in[\rho_-,\rho_+]$, $\rho\neq0$.
From now on, we assume that $T(\CV)\subseteq\PK$. In this case there exists an interval of ``admissible parameters'' $[\rho_-,\rho_+]$ such that $T^\#T + \rho V^\#V\in \mc{L}(\HH)^+$ if $\rho\in [\rho_-,\rho_+]$. If $\rho_-=\rho_+ =0$ then Problem \ref{pb smooth} becomes trivial. Therefore, we also assume that $\rho_-\neq 0$ or $\rho_+\neq 0$, and we consider a fixed $\rho\in[\rho_-,\rho_+]$, $\rho\neq0$.
%If $T(\CV)\subseteq\PK$, then there exists an interval of ``admissible parameters'' $[\rho_-,\rho_+]$ such that $T^\#T + \rho V^\#V\in \mc{L}(\HH)^+$ if $\rho\in [\rho_-,\rho_+]$.
%If $\rho_-=\rho_+ =0$ then Problem \ref{pb smooth} becomes trivial. 
%Therefore, from now on we consider a fixed $\rho\neq0$, and, as a result, if $T^\#T+\rho V^\#V\in\mc{L}(\HH)^+$ then $\rho_-\neq0$ or $\rho_+\neq0$, and $\rho\in[\rho_-,\rho_+]$.

\smallskip

Rewriting the discussion after Problem \ref{pb smooth 1} in terms of the original data $T$, $V$ and $\rho$, the following characterization of the existence of solutions to the indefinite abstract smoothing problem is obtained, see also \cite[Lemma 4.3]{GMMP10}. 

\begin{prop}\label{cond nec z0}
Given $z_0\in \EE$, Problem \ref{pb smooth} admits a solution if and only if $T^\#T + \rho V^\#V\in \mathcal{L}(\HH)^+$ and $V^\#z_0\in R(T^\#T + \rho V^\#V)$.
\end{prop}

\begin{proof}
As it was mentioned above, Problem \ref{pb smooth} admits a solution if and only if $R(L)$ is nonnegative and $(0,z_0)\in R(L)+R(L)^{\ort}$. Proposition \ref{prop cono positivo} shows that
$R(L)$ is nonnegative if and only if $L^\#L=T^\#T+ \rho V^\#V\in \mathcal{L}(\HH)^+$.

Also, $(0,z_0)\in R(L)+R(L)^{\ort}$ if and only if there exists $\tilde{x}\in\HH$ such that 
\[
\K{L\tilde{x} - (0,z_0)}{Lx}_\rho=0, \quad \text{for every $x\in \HH$},
\]
or equivalently, $L^\#(L\tilde{x})= L^\#(0,z_0)=\rho V^\#z_0$. Thus, 
\begin{equation*}
(0,z_0)\in R(L)+R(L)^{\ort}\quad \text{ if and only if } \quad V^\#z_0\in R(T^\#T + \rho V^\#V). \qedhere
\end{equation*}
\end{proof}

\begin{Def}
Any element $\tilde{x}\in \HH$ that is a solution to Problem \ref{pb smooth} is called a \emph{$(T,V,\rho)$-smoothing spline} to $z_0\in\EE$. 
The set of $(T,V,\rho)$-smoothing splines to $z_0$ is denoted by $sm(T,V,\rho,z_0)$.
\end{Def}

Since $T$ and $V$ are fixed along this work, $sm(T,V,\rho,z_0)$ is shortened to $sm(\rho,z_0)$.
%For short, the set of $(T,V,\rho)$-smoothing splines to $z_0\in\EE$ will be denoted by $sm(\rho,z_0)$.

\begin{cor}\label{ec normal}
Assume that $sm(\rho,z_0)\neq \varnothing$ for a vector $z_0\in\EE$. Then, $\tilde{x}\in sm(\rho,z_0)$ if and only if $\tilde{x}$ is a solution to the equation:
\begin{equation}\label{eq normal}
(T^\#T + \rho V^\#V)x= \rho V^\#z_0.
\end{equation}
In this case,
\[
sm(\rho,z_0)= \tilde{x}+ N(T^\#T + \rho V^\#V).
\]
\end{cor}

If the conditions established in Proposition \ref{cond nec z0} are satisfied, then 
\begin{equation}\label{solucion_daga}
\tilde{x}:=\rho (T^\#T + \rho V^\#V)^\dag V^\#z_0 \in sm(\rho,z_0).
\end{equation}
%From Proposition \ref{cond nec z0}, i
If $T^\#T+ \rho V^\#V\in\mc{L}(\HH)^+$, the smoothing problem admits a solution for those $z_0\in\EE$ such that $V^\#z_0\in R(L^\#L)$,
i.e. the set of admissible points for the indefinite abstract smoothing problem is
\[
\mc{A}:=(V^\#)^{-1}(R(L^\#L)).
\]
It is easy to see that the set of admissible points is closed if and only if the  subspace $R(L^\#L)\cap R(V^\#)$ is closed in $\HH$.

Since the inclusion $(V^\#)^{-1}(R(L^\#L))\subseteq V(N(L^\#L))^{\ort}$ holds,
in order to have a well-posed problem it is necessary that 
\[
z_0\in V(N(T^\#T+ \rho V^\#V))^{\ort},
\]
i.e. the set of admissible points $\mc{A}$ is contained in $V(N(T^\#T+ \rho V^\#V))^{\ort}$. 

\medskip 

The following lemmas lead to characterizing the case in which the set of admissible
points coincides with this closed subspace. 

\begin{lema}\label{admisibles}
Let $(y,z)\in \KK\x\EE$. Then, $(y,z)\in \ol{R(L) + R(L)^{\ort}}$ if and only if  
\[
z-VT^\dag y\in V(N(T^\#T + \rho V^\#V))^{\ort}.
\] 
\end{lema}

\begin{proof}
First, we prove that
\begin{equation}\label{parte isotropica}
R(L)^\circ= L(N(L^\#L)).
\end{equation}
In fact, the inclusion $L(N(L^\#L))\subseteq R(L)\cap N(L^\#)=R(L)^\circ$ is straightforward. To obtain the other inclusion, it suffices to apply $L$ to both sides
of the inclusion $L^{-1}(N(L^\#))\subseteq N(L^\#L)$.

Using that $T^\#$ is injective and $(T^\#)^\dag=(T^\dag)^\#$, is easy to see that $(T^\#T + \rho V^\#V)x=0$ if and only if $Tx=-\rho(T^\dag)^\#V^\#Vx$. Hence,
\begin{equation*}
R(L)^\circ=\set{(-\rho(T^\dag)^\#V^\#Vx,Vx)\in \KK\x\EE: \ x\in N(L^\#L)}.
\end{equation*}

Now, since $\ol{R(L) + R(L)^{\ort}}=(R(L)^\circ)^{\ort}$, the above description of $R(L)^\circ$ implies that $(y,z)\in\ol{R(L) + R(L)^{\ort}}$ if and only if 
\[
\K{(y,z)}{(-\rho(T^\dag)^\#V^\#Vx,Vx)}_\rho=0 \ \ \ \text{for every $x\in N(L^\#L)$}.
\]
But $\K{(y,z)}{(-\rho(T^\dag)^\#V^\#Vx,Vx)}_\rho=\K{y}{-\rho(T^\dag)^\#V^\#Vx}_\KK + \rho\K{z}{Vx}_\EE=\rho\K{z-VT^\dag y}{Vx}_\EE$. Hence, $(y,z)\in\ol{R(L) + R(L)^{\ort}}$ if and
only $\K{z-VT^\dag y}{Vx}_\EE=0$ for every $x\in N(L^\#L)$. Therefore, the assertion is proved.
\end{proof}

\begin{lema}\label{rango cdo}
The following conditions are equivalent:
\begin{enumerate}
	\item $R(L)$ is a closed subspace of $(\KK\x\EE, \K{\cdot}{\cdot}_\rho)$;
	\item $T(N(V))$ is a closed subspace of $\KK$;
	\item $N(T) + N(V)$ is a closed subspace of $\HH$.
\end{enumerate}
\end{lema}

\begin{proof}
(1)$\to$(2):\ Assume that $R(L)$ is a closed subspace. Then, since $L(N(V))=T(N(V))\x\set{0}$, $T(N(V))$ is closed if and only if $L(N(V))$ is closed, or equivalently, if $N(V)+N(L)$ is closed, see Proposition \ref{R(AB)}. But, the last condition is fulfilled because
$N(L)=N(T)\cap N(V)\subseteq N(V)$.
\smallskip

\noi The equivalence between (2) and (3) also follows from Proposition \ref{R(AB)}.
\smallskip

\noi (3)$\to$(1):\ Recall that $N(T)+N(V)$ is closed if and only if $N(T)^\bot+N(V)^\bot$ is closed, see \cite[Lemma 11]{Deutsch}. If we assume that $N(T)+N(V)$ is closed,  then $R(L^\#)=N(T)^\bot+N(V)^\bot$ is closed, and consequently $R(L)$ is closed.   
\end{proof}

Now we are in conditions to characterize the case in which the set of admissible points coincides with the subspace $V(N(L^\#L))^{\ort}$.  

\begin{prop}\label{prop_admisibles_cerrado}
Assume that $T^\#T+\rho V^\#V\in \mc{L}(\HH)^+$. Then, the following conditions are equivalent:
\begin{enumerate}
\item $\mc{A}=V(N(T^\#T+\rho V^\#V))^{\ort}$;
\item $R(L)+R(L)^{\ort}$ is a closed subspace of $(\KK\x\EE,\K{\cdot}{\cdot}_\rho)$.
\end{enumerate}
\end{prop}

\begin{proof}
Assume that $\mc{A}=V(N(L^\#L))^{\ort}$.
%is closed, let us show that $R(L^\#L)\cap R(V^\#)$ is closed. 
%Let $(y_n)_{n\geq 1}$ be a sequence in $R(L^\#L)\cap R(V^\#)$ such that $y_n\to y_0$ for some $y_0\in\HH$.
%Note that $R(V^\#)$ is closed (because $V$ is surjective) and in particular $y_0\in R(V^\#)$. Also,
%\[
%(V^\#)^\dag y_n \ra (V^\#)^\dag y_0.
%\]
%But $((V^\#)^\dag y_n)_{n\geq 1}$ is a sequence in $\mc{A}$, therefore $(V^\#)^\dag y_0\in\mc{A}$. Hence, $y_0=V^\#(V^\#)^\dag y_0\in R(L^\#L)\cap R(V^\#)$ and $R(L^\#L)\cap R(V^\#)$ is a closed subspace of $\HH$.
%Since $R(L^\#L)\cap R(V^\#)$ is closed, it holds that $\mc{A}=V(N(L^\#L))^{\ort}$. Now, let us prove that $R(L)+R(L)^{\ort}$ is closed.
Given $(y,z)\in \overline{R(L)+R(L)^{\ort}}$, Lemma \ref{admisibles} implies that $z-VT^\dag y\in\mc{A}$, or equivalently, the indefinite least-squares problem
consisting of minimizing $\K{Lx-(y,z)}{Lx-(y,z)}$ over $x\in\HH$ admits a solution.
In this case $(y,z)\in R(L)+ R(L)^{\ort}$, see \cite[Thm. 8.4]{Bognar}. Therefore, $R(L)+ R(L)^{\ort}$ is closed.
%Let $(y_n)_{n\in\NN}\subseteq R(L^\#L)\cap R(V^\#)$ and $y_0\in\HH$ be such that $y_n\to y_0$.
%Then, there exists $(x_n)_{n\in\NN}\subseteq (V^\#)^{-1}(R(L^\#L))$ such that $V^\#x_n=y_n$, for every $n\in\NN$. Also, since $R(V^\#)=N(V)^\bot$ is closed, there exists $x_0\in\EE$ such that $V^\#x_0=y_0$.
%Since $(V^\#)^\dag V^\#=I$, it follows that $x_n=(V^\#)^\dag V^\#x_n=(V^\#)^\dag y_n$, and $x_0=(V^\#)^\dag V^\#x_0=(V^\#)^\dag y_0$. Thus, $x_n\to x_0$, and since $(V^\#)^{-1}(R(L^\#L))=\mc{A}$ is
%closed, it holds that $x_0\in (V^\#)^{-1}(R(L^\#L))$. Hence, $y_0=V^\#x_0\in R(L^\#L)\cap R(V^\#)$. 
%
%Since $R(L^\#L)\cap R(V^\#)$ is closed, it holds that $\mc{A}=V(N(L^\#L))^{\ort}$. Now let $(y,z)\in \overline{R(L)+R(L)^{\ort}}$. Lemma \ref{admisibles} then implies that $z-VT^\dag y\in\mc{A}$, and
%consequently $sm(\rho,z-VT^\dag y)\neq\varnothing$. Let $\widetilde{x}_0\in sm(\rho,z-VT^\dag y)$, then, for every $x\in\HH$
%\begin{align*}
%&\K{L(x+T^\dag y)-(y,z)}{L(x+T^\dag y)-(y,z)}_\rho=\\
%&=\K{Lx-(0,z-VT^\dag y)}{Lx-(0,z-VT^\dag y)}_\rho\\
%&\geq \K{L\widetilde{x}_0-(0,z-VT^\dag y)}{L\widetilde{x}_0-(0,z-VT^\dag y)}_\rho\\
%&=\K{L(\widetilde{x}_0+T^\dag y)-(y,z)}{L(\widetilde{x}_0+T^\dag y)-(y,z)}_\rho.
%\end{align*}
%Since the indefinite least squares problem consisting of minimizing $\K{Lx-(y,z)}{Lx-(y,z)}_\rho$ over $x\in\HH$ admits a solution, it follows that $(y,z)\in R(L)+R(L)^{\ort}$.

Conversely, assume that $R(L)+R(L)^{\ort}$ is closed.
%In order to see that $\mc{A}$ is closed, we will proceed 
%It is sufficient to show that $\mc{A}=V(N(L^\#L))^{\ort}$.
First, the inclusion $\mc{A}\subseteq V(N(L^\#L))^{\ort}$ always holds. To see the other inclusion, if $z_0\in V(N(L^\#L))^{\ort}$ then, by Lemma \ref{admisibles}, it follows that
$(0,z_0)\in R(L)+R(L)^{\ort}$, i.e. $sm(\rho,z_0)\neq\varnothing$. Hence, $z_0\in\mc{A}$.
\end{proof}

As an immediate corollary, the following characterizes the conditions under which $R(L)$ is a pseudo-regular subspace.

\begin{cor}\label{todo admisible}
Assume that $T^\#T+\rho V^\#V\in \mc{L}(\HH)^+$ and $N(T) + N(V)$ is a closed subspace of $\HH$. Then, the following conditions are equivalent:
\begin{enumerate}
	\item $sm(\rho,z_0)\neq \varnothing$ for every $z_0\in V(N(T^\#T + \rho V^\#V))^{\ort}$;
	\item $R(L)$ is a pseudo-regular subspace of $(\KK\x\EE,\K{\cdot}{\cdot}_\rho)$;
	\item $T^\#T + \rho V^\#V$ has closed range. % and $N(T) + N(V)$ is closed in $\HH$.
\end{enumerate}
\end{cor}

\begin{proof} 
First of all, note that Lemma \ref{rango cdo} implies that $R(L)$ is a closed subspace.
The equivalence between (1) and (2) then follows from Proposition \ref{prop_admisibles_cerrado}, and applying Proposition \ref{preeliminares pseudo regular} yields the equivalence between (2) and (3).
\end{proof}

\medskip

%As a consequence of Theorem \ref{todo admisible}, t
The following proposition describes necessary and sufficient conditions for the existence of smoothing splines to every element of $\EE$,
see \cite[Prop. 4.5]{GMMP10}.

\begin{prop}
Assume that $T^\#T+\rho V^\#V\in \mc{L}(\HH)^+$. Then, the following conditions are equivalent:
\begin{enumerate}
	\item $sm(\rho,z_0)\neq\varnothing$ for every $z_0\in \EE$;
	\item $R(L)$ is a regular subspace of $(\KK\x\EE,\K{\cdot}{\cdot}_\rho)$;
	\item $R(T^\#T + \rho V^\#V)=N(T)^\bot + N(V)^\bot$.
\end{enumerate}

In this case, if $\widetilde{x}\in sm(\rho,z_0)$ then
\[
sm(\rho,z_0)=\widetilde{x}+N(T)\cap N(V).
\]
\end{prop}

\begin{proof}
(1)$\to$(3):\ If $sm(\rho,z_0)\neq\varnothing$ for every $z_0\in \EE$, then $(V^\#)^{-1}(R(L^\#L))=\EE$. This implies that $R(V^\#V)=R(V^\#)\subseteq R(L^\#L)$. Therefore, $R(L^\#L)=R(T^\#T) + R(V^\#V)=N(T)^\bot + N(V)^\bot$.
\smallskip

\noi (3)$\to$(2):\ If $R(L^\#L)=N(T)^\bot + N(V)^\bot=R(L^\#)$, taking the counterimage by $L^\#$ to both sides of the equality gives
\[
\HH=(L^\#)^{-1}(R(L^\#L))= R(L) + N(L^\#)=R(L) + R(L)^{\ort}.
\]
Therefore, $R(L)$ is regular.
\smallskip

\noi (2)$\to$(1):\ If $R(L) + R(L)^{\ort}=\HH$, applying $L^\#$ to both sides of the equality we get $R(L^\#L)=R(L^\#)=N(T)^\bot + N(V)^\bot$. Then,
\[
(V^\#)^{-1}(R(L^\#L))\supseteq  (V^\#)^{-1}(N(V)^\bot)=\EE. 
\]
Therefore,  $sm(\rho,z_0)\neq\varnothing$ for every $z_0\in \EE$.
\end{proof}

\medskip

\section{Relationship with the linearly constrained indefinite abstract interpolation problem}\label{semidef}

In this section a particular version of the constrained interpolation problem formulated in \eqref{QCQP} is analyzed, one which is closely related to the smoothing problem that arises from its
regularization. If the quadratic constraint in \eqref{QCQP} is replaced by a linear constraint, then Problem \ref{QCQP} translates into the following:

\begin{prob}\label{pb splines}
Given $z_0\in \EE$,
\begin{equation*}
\textit{\normalfont{minimize }}\K{Tx}{Tx},\textit{ \normalfont{subject to} }Vx=z_0,
\end{equation*}
and if the minimum exists, find the set of arguments at which it is attained.
\end{prob}

Note that Problem \ref{QCQP} in fact reduces to this problem if $V^\#V$ is  (positive or negative) semidefinite.

This problem was already considered in \cite{GMMP10} where, under some hypotheses, it was proved that the set of solutions to Problem \ref{pb splines} for a given $z_0\in\EE$ coincides with
$sm(\rho,z_0')$ for another $z_0'\in\EE$. Hereafter, we propose a deeper insight in this relationship. In particular, we study situations in which
the solutions to Problem \ref{pb splines} are contained in a set of $(T,V,\rho)$-smoothing splines, and the cases when Problem \ref{pb splines} can be translated into an abstract indefinite smoothing problem and viceversa.

\begin{Def}
Any element $\tilde{x}\in \HH$ that is a solution to Problem \ref{pb splines} is called a \emph{$(T,V)$-interpolating spline} to $z_0\in\EE$. 
The set of $(T,V)$-interpolating splines to $z_0$ is denoted by $sp(T,V,z_0)$.
\end{Def}

Since $T$ and $V$ are fixed along this work, $sp(T,V,z_0)$ is shortened to $sp(z_0)$.
%For short, the set of $(T,V)$-interpolating splines to $z_0\in\EE$ will be denoted by $sp(z_0)$.
\smallskip

Necessary and sufficient conditions for the existence of solutions to Problem \ref{pb splines} are given in the following proposition, which is an immediate consequence of \cite[Lemma 3.4]{GMMP10}.

\begin{prop}\label{prop sp no vacio}
Assume that $T(N(V))$ is nonnegative and let $z_0\in\EE$. Then, the following conditions are equivalent:
\begin{enumerate}
\item $sp(z_0)\neq\varnothing$;
\item $z_0\in V\parentesis{T^\#T(N(V))^\bot}$.
\end{enumerate}
In this case, $sp(z_0)$ is an affine manifold parallel to the subspace 
\[
\N_0:= N(V) \cap T^\#T(N(V))^\bot.
\]
 %then $\N_0$; moreover, given $x_0\in\HH$ and setting $z_0=Vx_0$,
%$x_0\in sp(z_0)$ if and only if $x_0\in T^\#T(N(V))^\bot$.
\end{prop}

\begin{obs}\label{obs_sol_sp}
Note that \cite[Lemma 3.4]{GMMP10} assures that, given $z_0\in\EE$, $x_0\in sp(z_0)$ if and only if $Vx_0=z_0$ and $x_0\in T^\#T(N(V))^\bot$.
\end{obs}

In \cite[Prop. 3.8]{GMMP10} necessary and sufficient conditions for the existence of interpolating splines to every vector of the space were provided:
$sp(z_0)\neq\varnothing$ for every $z_0\in\EE$ if and only if $T(N(V))$ is a (closed) uniformly positive subspace of $\KK$. This is the case if
%,considering the operator $L:\HH\ra\KK\x\EE$ defined in \eqref{L}, 
$R(L)$ is a (positive) regular subspace
of $(\KK\x\EE,\K{\cdot}{\cdot}_\rho)$, see \cite[Thm. 4.7]{GMMP10}. Also, several properties of the subspace $T(N(V))$ were analyzed in order to characterize the interpolating splines. 

In the following we proceed to study a more general case.
%case when $R(L)$ is pseudo-regular. 
With this aim, we first describe the relationship between the subspaces $R(L)$ and $T(N(V))$.

%In \cite{GMMP10}
%The next lemma describes the features of $T(N(V))$ more thouroughly, which depicts their resemblance to the characteristics of $R(L)$. 

 \begin{lema}\label{props_tnv}
The following conditions hold:
\begin{enumerate}
	\item $T(N(V))^\circ=T(\N_0)$. Moreover, 
	\begin{equation}\label{tiso}
	T(N(V))= T(\N_0) \sdo\ T(N(V)\ominus \N_0);
	\end{equation}
	\item $T(N(V))$ is nondegenerate if and only if $\N_0=N(V)\cap N(T)$;
	\item $T(N(V))$ is pseudo-regular if and only if $N(V) + T^\#T(N(V))^\bot$ is a closed subspace of $\HH$.	
	In this case, $T(N(V)\ominus \N_0)$ is regular;
	\item $T(N(V))$ is regular if and only if $\HH=N(V) + T^\#T(N(V))^\bot$.
\end{enumerate}
\end{lema}

\begin{proof}
Most of the conditions in the statement can be derived from
\begin{equation}\label{idayvuelta}
	T^\#T(N(V))^\bot=T^{-1}(T(N(V))^{\ort}),
\end{equation}
which is a consequence of \eqref{preimag} in the Preliminaries.
\smallskip

\noi (1)\ \ If $y\in T(N(V))^\circ$, there exists $x\in N(V)$ such that $y=Tx\in T(N(V)^{\ort})$. So, \eqref{idayvuelta} implies that
$x\in \N_0$ and $y=Tx\in T(\N_0)$. The other inclusion is also a consequence of \eqref{idayvuelta}.

Also, since $N(V)=\N_0\oplus (N(V)\ominus \N_0)$ it is easy to see that $T(N(V))=T(\N_0) + T(N(V)\ominus \N_0)$. Moreover, these subspaces are $\K{\cdot}{\cdot}_\KK$-orthogonal because
$T(\N_0)\subseteq T(N(V))^{\ort}$ and $T(N(V)\ominus \N_0)\subseteq T(N(V))$. It remains to prove that the sum is direct, but if $y\in T(\N_0)\cap T(N(V)\ominus \N_0)$ then there exist
$x_1\in \N_0$ and $x_2\in N(V)\ominus \N_0$ such that $Tx_1=y=Tx_2$. So, $x_2-x_1\in N(V)\cap N(T)\subseteq \N_0$ and $x_2=(x_2-x_1) + x_1\in \N_0\cap (N(V)\ominus \N_0)=\{0\}$. Thus,
$T(N(V))=T(\N_0) \sdo\ T(N(V)\ominus \N_0)$.
\smallskip

\noi (2)\ \ Observe that $N(V)\cap N(T)\subseteq \N_0$. Also, $T(N(V))$ is nondegenerate if and only if $T(\N_0)=\{0\}$, or equivalently, if $\N_0\subseteq N(T)$. This completes
the proof of the assertion.
\smallskip

\noi (3)\ \ Assume that $T(N(V)) + T(N(V))^{\ort}$ is closed in $\KK$. Since $T$ is bounded, 
\begin{align}\label{preim}
T^{-1}(T(N(V)) + T(N(V))^{\ort})&= T^{-1}(T(N(V))) + T^{-1}(T(N(V))^{\ort})=\nonumber\\
&=(N(V) + N(T)) + T^\#T(N(V))^\bot\nonumber\\
&= N(V) + T^\#T(N(V))^\bot
\end{align}
is also a closed subspace. 

Conversely, assume that $N(V) + T^\#T(N(V))^\bot$ is closed in $\HH$. Let $(y_n)_{n\in\NN}$ be a sequence in $T(N(V)) + T(N(V))^{\ort}$ such that $y_n\ra y$ for some $y\in\KK$. Since $T$ is
surjective, there exists $x\in N(T)^\bot$ such that $Tx=y$. Also, by \eqref{preim}, for each $n\geq 1$ there exists $x_n\in N(V) + T^\#T(N(V))^\bot$ such that $Tx_n=y_n$. Since $T^\dag$ is bounded,
\begin{equation}\label{lala}
P_{N(T)^\bot}x_n=T^\dag Tx_n =T^\dag y_n \ra T^\dag y=T^\dag Tx= x.
\end{equation}
Since $N(T)\subseteq T^\#T(N(V))^\bot$ it is immediate that $N(T) + (N(V) + T^\#T(N(V))^\bot)$ is a closed subspace. Equivalently, $P_{N(T)^\bot} (N(V) + T^\#T(N(V))^\bot)$ is also closed in $\HH$. By
\eqref{lala}, it follows that $x\in P_{N(T)^\bot} (N(V) + T^\#T(N(V))^\bot)$. So, there exists $z\in N(V) + T^\#T(N(V))^\bot$ such that $x= P_{N(T)^\bot} z$. Thus,
$y=Tx=TP_{N(T)^\bot} z=Tz\in T(N(V)) + T(N(V))^{\ort}$, see \eqref{preim}. Therefore, $T(N(V)) + T(N(V))^{\ort}$ is closed in $\KK$.

It remains to show that if $T(N(V))$ is pseudo-regular then $T(N(V)\ominus \N_0)$ is regular. 
%Since 
%\[
%T(N(V))=T(N(V))^\circ\ \sdo \ T(N(V)\ominus \N_0)
%\]
Since $T(\N_0)=T(N(V))^\circ$ and  \eqref{tiso},
it suffices to prove that $T(N(V)\ominus \N_0)$ is closed. 

By Proposition \ref{R(AB)}, $T(N(V)\ominus \N_0)$ is closed if and only if $(N(V)\ominus \N_0) + N(T)$ is closed. By \eqref{preim},
$N(V) + T^\#T(N(V))^\bot=(N(V)\ominus \N_0) \oplus T^\#T(N(V))^\bot$ is closed. Since $N(T)\subseteq T^\#T(N(V))^\bot$, then
$(N(V)\ominus \N_0) + N(T)$ is also closed (and regular). Thus, the proof is complete.

\noi (4)\ \ Assume that $T(N(V))$ is regular. Then, \eqref{preim} implies that
\begin{align*}
\HH & = T^{-1}(\KK)=T^{-1}(T(N(V)) + T(N(V))^{\ort})= N(V) + T^\#T(N(V))^\bot.
\end{align*}
Conversely, assume that $\HH=N(V) + T^\#T(N(V))^\bot$. Since $T$ is surjective, 
\begin{align*}
\KK &=T(N(V) + T^\#T(N(V))^\bot)= T(N(V)) + T(T^{-1}(T(N(V))^{\ort}))\nonumber\\
&=T(N(V)) + T(N(V))^{\ort},
\end{align*}
i.e. $T(N(V))$ is regular.
\end{proof}

The following lemma and proposition continue describing the relationship between the subspaces $R(L)$ and $T(N(V))$.

\begin{lema}\label{lema_tnv_no_degenerado} If $R(L)$ is a (closed) nonnegative subspace of $(\KK\x\EE,\K{\cdot}{\cdot}_\rho)$ then $T(N(V))$ is a (closed) positive subspace of $\KK$. In particular,  $T(N(V))$ is nondegenerate.
\end{lema}

\begin{proof} Assume that  $R(L)$ is a closed nonnegative subspace of $(\KK\x\EE,\K{\cdot}{\cdot}_\rho)$. 
%As a result, as stated in the proof of Proposition \ref{todo admisible}
By Lemma \ref{rango cdo}, $T(N(V))$ is a closed subspace of $\KK$. Also, since $T(N(V))\x\{0\}=L(N(V))$ is a subspace of $R(L)$,
it follows that $T(N(V))$ is nonnegative.

In order to prove that $T(N(V))$ is positive, it is sufficient to show that $T(N(V))$ is nondegenerate. So, let $x\in N(V)$ such that $Tx\in T(N(V))^\circ$. Since $x\in N(V)$, %$T^\#Tx=L^\#Lx$ and
\[
0=\K{Tx}{Tx}=\PI{Lx}{Lx}.
\]
Since $L^\#L\in L(\HH)^+$ (see Proposition \ref{prop cono positivo}) it follows that $T^\#Tx=L^\#Lx=0$. Finally, the surjectivity of $T$ implies the injectivity of $T^\#$ and, in particular, $Tx=0$. 
\end{proof}

\begin{prop}\label{lema relacion entre R(L) y T(N(V))}
If $R(L)$ is a nonnegative pseudo-regular subspace of $(\KK\x\EE,\K{\cdot}{\cdot}_\rho)$ then $T(N(V))$ is a (closed) uniformly positive subspace of $\KK$. 
\end{prop}

\begin{proof}
The idea of the proof is to show that $T(N(V))\x\{0\}$ is contained in a regular complement of $R(L)^\circ$ in $R(L)$. First, observe
that 
%$R(L)^\circ\cap (T(N(V))\x\{0\})=\{0\}$. In fact, it is a consequence of 
$R(L)^\circ$ is a neutral subspace and that, by Lemma \ref{lema_tnv_no_degenerado}, $T(N(V))\x\{0\}$ is a closed positive subspace. Hence, $R(L)^\circ\cap (T(N(V))\x\{0\})=\{0\}$.
%if $(y,z)\in R(L)^\circ\cap (T(N(V))\x\{0\})$ then there exist $x\in N(T^\#T + \rho V^\#V)$ and $w\in N(V)$ such that
%\[
%(Tx,Vx)=(y,z)=(Tw,0),
%\] 
%see Lemma \ref{isotr}. Hence, $x\in N(T^\#T + \rho V^\#V)\cap N(V)=N(T^\#T)=N(T)=\{0\}$.
Also, $T(N(V))\x\{0\}=L(N(V))$ is $\K{\cdot}{\cdot}_\rho$-orthogonal to $R(L)^\circ$.
% because it is a subspace of $R(L)$.

%Finally, note that
%\[
%R(L)^\circ \sdo \ (T(N(V))\x\{0\})=L(N(T^\#T + \rho V^\#V)) + L(N(V))=L(N(T^\#T + \rho V^\#V) + N(V)),
%\]
%which is closed by Proposition \ref{R(AB)}. 

Therefore, there exists a closed subspace $\M$ of $R(L)$ such that 
\[
R(L)=R(L)^\circ \sdo \ \M \ \ \ \text{and} \ \ \ T(N(V))\x\{0\}\subseteq \M.
\]
Since $R(L)$ is pseudo-regular and nonnegative, the subspace $\M$ is a uniformly positive subspace of $\KK\x\EE$ and $T(N(V))\x\{0\}$ has the same property (by transitivity). Thus, $T(N(V))$ is uniformly positive.
\end{proof}

Combining \cite[Prop. 3.8]{GMMP10} with the above result,
%As a result, 
we can now ensure the existence of $(T,V)$-interpolating splines to every element of $\EE$ in case that $R(L)$ is a (nonnegative) pseudo-regular subspace.

\begin{cor} If $R(L)$ is a nonnegative pseudo-regular subspace of $(\KK\x\EE,\K{\cdot}{\cdot}_\rho)$, then 
\[
sp(z_0)\neq\varnothing \quad \text{for every $z_0\in\EE$}.
\]
\end{cor}

\medskip

We now focus our attention on the situations in which a set of interpolating splines is also a subset of the solutions to a certain smoothing interpolation problem, or, even further,
in which an interpolating splines problem can be posed as a smoothing problem, and viceversa. 
%The following proposition presents a first result, see \cite[Thm. 4.7]{GMMP10}.
We start by presenting what is known so far, see \cite[Thm. 4.7]{GMMP10}.

\begin{prop} Suppose that $R(L)$ is a positive regular subspace of $(\KK\x\EE,\K{\cdot}{\cdot}_\rho)$.
Then, for every $z_0\in \EE$ there exists $w_0\in\EE$ such that
\[
sp(w_0)=sm(\rho,z_0).
\]
\end{prop}

Although this result is interesting (it says that every $(T,V)$-interpolating spline is a $(T,V,\rho)$-smoothing spline and viceversa), the hypothesis under consideration is quite strong. The rest of the manuscript is
devoted to some intermediate results by relaxing the conditions imposed on $R(L)$.

\begin{prop}\label{prop_inclusion} Suppose that $R(L)$ is a nonnegative subspace of $(\KK\x\EE,\K{\cdot}{\cdot}_\rho)$. Then, the following conditions hold:
\begin{enumerate}
\item for every $z_0\in\EE$ such that $sm(\rho,z_0)\neq\varnothing$ there exists $w_0\in\EE$ such that
\[
\varnothing\neq sp(w_0)\subseteq sm(\rho,z_0);
\]
\item for every $w_0\in\EE$ such that $sp(w_0)\neq\varnothing$ there exists $z_0\in\EE$ such that
\[
sp(w_0)\subseteq sm(\rho,z_0).
\]
\end{enumerate}
\end{prop}

\begin{proof}
(1)\ \  Assume that $x_0\in sm(\rho,z_0)$. By Proposition \ref{ec normal} it follows that $(T^\#T+\rho V^\#V)x_0=\rho V^\#z_0$. Since $T^\#Tx_0=V^\#(\rho z_0-Vx_0)\in N(V)^\bot$,
it holds that $x_0\in T^\#T(N(V))^\bot$.

Also, Lemma \ref{lema_tnv_no_degenerado} implies that $T(N(V))$ is a positive and nondegenerate subspace of $\KK$. But then $N(V)\cap T^\#T(N(V))^\bot=N(V)\cap N(T)$, see Lemma
\ref{props_tnv}.
%which in turn assures by Lemma \ref{props TS} that $N(V)\cap T^\#T(N(V))^\bot=N(V)\cap N(T)$.
Setting $w_0:=Vx_0$, by Remark \ref{obs_sol_sp} it follows that $x_0\in sp(w_0)$, and thus
\[
sp(w_0)=x_0+N(V)\cap N(T).
\]
Finally, %since $N(V)\cap N(T)\subseteq N(L^\#L)$ it follows that
%\begin{equation*}
$sp(w_0)=x_0+N(V)\cap N(T)\subseteq x_0+N(L^\#L)=sm(\rho,z_0)$.
%\end{equation*}

\smallskip

\noi (2)\ \ Note that $T(N(V))$ is a nonnegative subspace of $\KK$ because $R(L)$ is a nonnegative subspace of $(\KK\x\EE,\K{}{}_\rho)$. Indeed,
$L(N(V))=T(N(V))\x\set{0}$. 

Assume that $sp(w_0)\neq\varnothing$ for some $w_0\in\EE$.
By Remark \ref{obs_sol_sp}, there exists $x_0\in T^\#T(N(V))^\bot$ such that $Vx_0=w_0$. Since $T^\#Tx_0\in N(V)^\bot=R(V^\#)$, it follows that
$T^\#Tx_0=V^\#(V^\#)^\dag T^\#T x_0$. Then,
\[
(T^\#T+\rho V^\#V)x_0=\rho V^\#\parentesis{\frac{1}{\rho}(V^\#)^\dag T^\#T+V}x_0.
\]
Thus, denoting $z_0:=\parentesis{\frac{1}{\rho}(V^\#)^\dag T^\#T+V}x_0$, Corollary \ref{ec normal} assures that $x_0\in sm(\rho,z_0)$. Since $N(V)\cap N(T)\subseteq N(L^\#L)$, the result follows.
\end{proof}

Note that given $z_0\in\EE$ such that $sm(\rho,z_0)\neq\varnothing$, by \eqref{solucion_daga}, setting $w_0:=\rho V(T^\#T+\rho V^\#V)^\dag V^\#z_0$ yields $\varnothing\neq sp(w_0)\subseteq sm(\rho,z_0)$.
Analogously, given
$w_0\in \EE$ such that $sp(w_0)\neq\varnothing$, if $x_0\in T^\#T(N(V))^\bot$ is such that $w_0=Vx_0$, then setting $z_0:=\parentesis{\frac{1}{\rho}(V^\#)^\dag T^\#T+V}x_0$ yields
$sp(w_0)\subseteq sm(\rho,z_0)$.

\medskip

As an immediate consequence of Proposition \ref{prop_inclusion}, the following corollary holds:

\begin{cor} Suppose that $R(L)$ is a positive nondegenerate subspace of $(\KK\x\EE,\K{\cdot}{\cdot}_\rho)$. Then, the following conditions hold:
\begin{enumerate}
\item for every $z_0\in\EE$ such that $sm(\rho,z_0)\neq\varnothing$ there exists $w_0\in\EE$ such that
\[
\varnothing\neq sp(w_0)=sm(\rho,z_0);
\]
\item for every $w_0\in\EE$ such that $sp(w_0)\neq\varnothing$ there exists $z_0\in\EE$ such that
\[
sp(w_0)=sm(\rho,z_0).
\]
\end{enumerate}
\end{cor}

\begin{proof}
By \eqref{parte isotropica}, $R(L)^\circ=L(N(L^\#L))$. Hence, $R(L)$ is nondegenerate if and only $N(L^\#L)\subseteq N(L)=N(T)\cap N(V)$. In this case, $N(L^\#L)=N(T)\cap N(V)$, and the
statement then follows from Proposition \ref{prop_inclusion}.
\end{proof}

\end{document}